\documentclass[a4paper,12pt,reqno]{amsart}

\usepackage{amsmath, amsfonts, amssymb, amsthm, mathrsfs, mathtools, mathalfa}
\usepackage{comment}
\usepackage{yfonts}
\usepackage{enumerate}
\usepackage{enumitem}
\usepackage{stmaryrd}
\usepackage{fancyhdr}
\usepackage{bm}
\usepackage[utf8]{inputenc}
\usepackage[pdfencoding=auto, hyperfootnotes=false]{hyperref}
\usepackage{tikz,tikz-cd}
\usepackage[hang,flushmargin]{footmisc}
\usepackage{graphicx}
\usepackage{xr}
\usepackage{extarrows}
\usepackage[dvistyle, colorinlistoftodos]{todonotes}

\usepackage{tikz}
\usetikzlibrary{matrix}
\tikzset{stretch/.initial=1}
\usetikzlibrary{calc}

\newcommand\drawloop[4][]%
   {\draw[shorten <=0pt, shorten >=0pt,#1]
      ($(#2)!\pgfkeysvalueof{/tikz/stretch}!(#2.#3)$)
      let \p1=($(#2.center)!\pgfkeysvalueof{/tikz/stretch}!(#2.north)-(#2)$),
          \n1= {veclen(\x1,\y1)*sin(0.5*(#4-#3))/sin(0.5*(180-#4+#3))}
      in arc [start angle={#3-90}, end angle={#4+90}, radius=\n1]%
   }

\makeatletter
\def\@tocline#1#2#3#4#5#6#7{\relax
  \ifnum #1>\c@tocdepth 
  \else
    \par \addpenalty\@secpenalty\addvspace{#2}%
    \begingroup \hyphenpenalty\@M
    \@ifempty{#4}{%
      \@tempdima\csname r@tocindent\number#1\endcsname\relax
    }{%
      \@tempdima#4\relax
    }%
    \parindent\z@ \leftskip#3\relax \advance\leftskip\@tempdima\relax
    \rightskip\@pnumwidth plus4em \parfillskip-\@pnumwidth
    #5\leavevmode\hskip-\@tempdima
      \ifcase #1
       \or\or \hskip 1em \or \hskip 2em \else \hskip 3em \fi%
      #6\nobreak\relax
    \dotfill\hbox to\@pnumwidth{\@tocpagenum{#7}}\par
    \nobreak
    \endgroup
  \fi}
\makeatother

\newtheorem{theorem}{Theorem}[section]
\newtheorem{lemma}[theorem]{Lemma}

\newtheorem{corollary}[theorem]{Corollary}
\newtheorem{proposition}[theorem]{Proposition}
\newtheorem{question}[theorem]{Question}

\theoremstyle{definition}
\newtheorem{defn}[theorem]{Definition}
\newtheorem{remark}[theorem]{Remark}

\newcommand{\mc}{\mathcal}

\newcommand{\mb}{\mathbb}

\DeclareMathOperator{\aff}{aff}

\DeclareMathOperator{\q}{c}

\DeclareMathOperator{\co}{\circ\hspace{-0.02 cm}}

\DeclareMathOperator{\tIm}{Im}
\begin{document}

\title{A refinement of Cauchy-Schwarz complexity}

\author{Pablo Candela}
\address{Universidad Aut\'onoma de Madrid and ICMAT\\ Departamento de Matem\'aticas, Universidad Aut\'onoma de Madrid   (office 212)
Ciudad Universitaria de Cantoblanco\\ Madrid 28049\\ Spain}
\email{pablo.candela@uam.es}

\author{Diego Gonz\'alez-S\'anchez${}^*$}
\address{MTA Alfr\'ed R\'enyi Institute of Mathematics, Re\'altanoda u. 13-15.\\
Budapest, Hungary, H-1053}
\email{diegogs@renyi.hu}
\thanks{${}^*$ Corresponding author.}

\author{Bal\'azs Szegedy}
\address{MTA Alfr\'ed R\'enyi Institute of Mathematics, Re\'altanoda u. 13-15.\\
Budapest, Hungary, H-1053}
\email{szegedyb@gmail.com}

\begin{abstract}
We introduce a notion of complexity for systems of linear forms called \emph{sequential Cauchy-Schwarz complexity}, which is parametrized by two positive integers $k,\ell$ and refines the notion of Cauchy-Schwarz complexity introduced by Green and Tao. We prove that if a system of linear forms has sequential Cauchy-Schwarz complexity at most $(k,\ell)$ then any average of 1-bounded functions over this system is controlled by the $2^{1-\ell}$-th power of the Gowers $U^{k+1}$-norms of the functions. For $\ell=1$ this agrees with Cauchy-Schwarz complexity, but for $\ell>1$ there are systems that have sequential Cauchy-Schwarz complexity at most $(k,\ell)$ whereas their Cauchy-Schwarz complexity is greater than $k$. Our main application illustrates this with systems over a prime field $\mb{F}_p$, denoted by $\Phi_{k,M}$, which can be viewed as $M$-dimensional  arithmetic progressions of length $k$. For each $M\geq 2$ we prove that $\Phi_{k,M}$ has sequential Cauchy-Schwarz complexity at most $(k-2,|\Phi_{k,M}|)$ (where $|\Phi_{k,M}|$ is the number of forms in the system), whereas the Cauchy-Schwarz complexity of $\Phi_{k,M}$ can be greater than $k-2$. Thus we obtain polynomial true-complexity bounds for $\Phi_{k,M}$ with exponent $2^{-|\Phi_{k,M}|}$. A recent general theorem of Manners, proved independently with different methods, implies a similar application but with different polynomial true-complexity bounds, as explained in the paper. In separate work, we use our application to give a new proof of the inverse theorem for Gowers norms on $\mb{F}_p^n$, and related results on ergodic actions of $\mb{F}_p^{\omega}$.
\end{abstract}
\maketitle

\section{introduction}

\subsection{Background}\hfill\\
\noindent Let $\mb{F}_p$ denote the finite field of prime order $p$. A system of $r$ linear forms in $d$ variables over $\mb{F}_p$ is a collection $\Psi=\{\psi_1,\ldots,\psi_r\}$ of linear maps $\psi_i:\mb{F}_p^d\to \mb{F}_p$. Such a system can also be viewed as a matrix in $\mb{F}_p^{r\times d}$ in which the $i$-th row vector $(\psi_{i,j})_{j\in [d]}\in \mb{F}_p^d$ corresponds to the form $\psi_i$ via the formula $\psi_i(x_1,\ldots,x_d):= \psi_{i,1}x_1+\cdots+\psi_{i,d} x_d$ for $x=(x_1,\ldots,x_d)\in \mb{F}_p^d$; in a slight abuse of notation, we shall denote this matrix also by $\Psi$. More generally, for any vector space $\mb{F}_p^n$ over $\mb{F}_p$, we let any vector $v\in \mb{F}_p^d$ (in particular any row $\psi_i$) act as the homomorphism $(\mb{F}_p^n)^d\to \mb{F}_p^n$, $(x_1,\ldots,x_d)\mapsto v(x_1,\ldots,x_d):= v_1x_1 + \cdots + v_d x_d$. Given complex-valued functions $f_1,\ldots,f_r$ on $\mb{F}_p^n$, we then define the \emph{$\Psi$-average} of these functions by the formula 
\begin{equation}\label{eq:psiav}
\Lambda_\Psi(f_1,\ldots,f_r) := \mb{E}_{x_1,\ldots,x_d \in \mb{F}_p^n} \prod_{i\in [r]} f_i\big(\psi_i(x_1,\ldots,x_d) \big).
\end{equation}
The analysis of such averages powers a large part of arithmetic combinatorics and related fields. In particular, if the functions $f_i$ are all equal to the indicator function $1_A$ of some set $A\subset \mb{F}_p^n$, the corresponding average $\Lambda_\Psi(1_A,\ldots,1_A)$ is the normalized count of configurations of the form $\Psi(x_1,\ldots,x_d)$ that are contained in $A$ (i.e.\ that have all of their $r$ components in $A$). These configuration counts can then be studied via the analysis of these averages. In arithmetic combinatorics, this analysis has focused on two regimes. The first and most classical one (often referred to as the \emph{integer setting}) considers linear forms over the integers, i.e.\ with each form $\psi_i$ being a homomorphism $\mb{Z}^d\to\mb{Z}$. Here the analysis is often carried out more conveniently in the setting of cyclic groups of large prime order. This was the setting in which, for instance, Gowers obtained his celebrated results on Szemer\'edi's theorem \cite{Go}, and then Green and Tao developed their programme toward counting linear configurations in the prime numbers \cite{GTlin}. In the present setup, this can be viewed as the regime where the dimension of $\mb{F}_p^n$ is fixed to $n=1$ and the prime $p$ is allowed to grow unbounded. The second principal regime, known as the \emph{finite field setting}, takes $p$ to be a fixed prime and instead allows the dimension $n$ to grow unbounded; for more background on this setting we refer to the survey \cite{WolfFF}.

In the analysis of multilinear averages of the form \eqref{eq:psiav}, the uniformity norms introduced by Gowers in \cite{Go} have become standard tools. Let us recall briefly that there is one such norm for each integer $k\geq 2$, called the \emph{$U^k$-norm}, defined on the space of complex-valued functions $f$ on $\mb{F}_p^n$ (it can be defined more generally for bounded Haar-measurable functions on any compact abelian group), and denoted by $\|f\|_{U^k}$. We refer to \cite[\S 11.1]{T&V} for the definition and basic background on these norms. 

Among the main facts that make the uniformity norms so useful in the analysis of $\Psi$-averages, there is a family of results that provide upper bounds, for a large class of such averages, in terms of the uniformity norms of the functions involved in the averages. The first example of such results is a theorem of Gowers \cite[Theorem 3.2]{Go} concerning averages over $k$-term arithmetic progressions (corresponding to the 2-variable system $\Psi=\{x_1,x_1+x_2,\ldots,x_1+(k-1)x_2\}$), which was crucial for the effective proof of Szemer\'edi's theorem given in \cite{Go}. Let us say that a complex-valued function $f$ is \emph{1-bounded} if the modulus $|f|$ is at most $1$ everywhere. In the present setup, this theorem of Gowers states that for every collection of 1-bounded functions $f_1,\ldots,f_k:\mb{F}_p^n\to\mb{C}$ with $k\leq p$, we have
\begin{equation}\label{eq:GVNorig}
\big|\mb{E}_{x_1,x_2 \in \mb{F}_p^n} f_1(x_1)f_2(x_1+x_2)\cdots f_k(x_1+(k-1)x_2)\big|\leq \min_{i\in [k]} \|f_i\|_{U^{k-1}}. 
\end{equation}
Gowers's proof of this estimate consisted in an iterated application of the Cauchy-Schwarz inequality combined with judicious changes of variables. This result was then extended from arithmetic progressions to a large class of systems by Green and Tao in \cite{GTlin}, where such estimates were named \emph{generalized non Neumann theorems}. A key ingredient for this extension was the introduction of a notion of \emph{complexity} for systems of linear forms, whereby, if a system $\Psi$ has complexity at most $k$, then any $\Psi$-average of 1-bounded functions can always be controlled by the $U^{k+1}$-norm of the functions, in the sense that there is a corresponding version of estimate \eqref{eq:GVNorig} for this system. To recall this in more detail, let us begin with the formal definition of this complexity notion, which was later called \emph{Cauchy-Schwarz complexity} by Gowers and Wolf; see \cite[Definition 1.1]{G&W1} (we give the definition in the present setting of $\mb{F}_p^n$).
\begin{defn}[Cauchy-Schwarz complexity over $\mb{F}_p$]\label{def:CS}
Let $\Psi$ be a system of $r$ linear forms in $d$ variables over $\mb{F}_p$. For $i\in [r]$, we say that $\Psi$ has \emph{Cauchy-Schwarz complexity at most $k$ at $i$ over $\mb{F}_p$} if the set of forms $\{\psi_j: j\in [r]\setminus\{i\}\}$ can be covered by $k+1$ subsets (or fewer) such that $\psi_i$ does not lie in the $\mb{F}_p$-linear-span of any of these subsets. If this holds for each $i\in [r]$, then we say that $\Psi$ has \emph{Cauchy-Schwarz complexity at most} $k$.
\end{defn}
\noindent We usually shorten the expression ``Cauchy-Schwarz complexity at most $k$ at $i$ over $\mb{F}_p$" to ``$\mb{F}_p$-CS-complexity $\leq k$ at $i$", or just to ``CS-complexity $\leq k$ at $i$" when the field $\mb{F}_p$ is clear. Let us write $s_{\textup{CS}(i)}(\Psi)$ for the smallest integer $k$ such that $\Psi$ has CS-complexity $\leq k$ at $i$, and let us write $s_{\textup{CS}}(\Psi)$ for the smallest integer $k$ such that $\Psi$ has CS-complexity $\leq k$. (Thus $s_{\textup{CS}(i)}$, $s_{\textup{CS}}$ implicitly depend on $p$.)

The result of Green and Tao extending the estimate \eqref{eq:GVNorig} was developed in a setting technically different from the present one, but it is not hard to extract the essence of their proof to obtain the following version of their result (see \cite[Theorem 2.3]{G&W1}, or \cite[Proposition 1.1]{Man} for a formulation closer to the following version).
\begin{theorem}\label{thm:basicGT}
Let $\Psi$ be a system of $r$ linear forms in $d$ variables over $\mb{F}_p$, let $i\in[r]$ and suppose that $s_{\textup{CS}(i)}(\Psi)\leq k$. Then for every collection of 1-bounded functions $f_1,\ldots,f_r:\mb{F}_p^n\to \mb{C}$ we have
\begin{equation}\label{eq:basicGT}
\big| \Lambda_\Psi(f_1,\ldots,f_r)\big| \leq \|f_i\|_{U^{k+1}}.
\end{equation}
\end{theorem}
\noindent Estimates of this kind are often used to reduce questions about counting linear configurations in sets to questions pertaining to the analysis of uniformity norms. This analysis becomes more intricate (and the associated bounds worsen) as $k$ increases. This motivates the following question, which was posed and investigated by Gowers and Wolf in \cite{G&W1}.
\begin{question}
Given a system $\Psi$ of $r$ linear forms, what is the least integer $k$ for which an estimate of the form \eqref{eq:basicGT} holds for every collection of $1$-bounded functions $f_i$, $i\in [r]$?
\end{question}
\noindent The study of this question in \cite{G&W1} started from the observation that there are systems for which the notion of complexity from Definition \ref{def:CS} does not yield the optimal answer to the question (i.e.\ the least integer $k$). Accordingly, Gowers and Wolf defined a notion of \emph{true complexity} for systems of linear forms, and coined the term \emph{Cauchy-Schwarz complexity} to distinguish the notion in Definition \ref{def:CS} from true complexity. Let us recall the definition of true complexity in the present setup.
\begin{defn}\label{def:TC}
Let $\Psi$ be a system of $r$ linear forms in $d$ variables over $\mb{F}_p$.  We say that $\Psi$ has \emph{true complexity at most $k$ at} $i\in [r]$ over $\mb{F}_p$ if there is a function $\varepsilon:\mb{R}_{>0}\to \mb{R}_{>0}$ such that $\varepsilon(\delta)\to 0$ as $\delta\to 0$, and such that for every collection of 1-bounded functions $f_1,\ldots,f_r:\mb{F}_p^n\to \mb{C}$ the following estimate holds:
\begin{equation}\label{eq:TCest-i}
\big| \Lambda_\Psi(f_1,\ldots,f_r)\big| \leq \varepsilon(\|f_i\|_{U^{k+1}}).
\end{equation}
We write $s_{(i)}(\Psi)$ for the least integer $k$ such that $\Psi$ has true complexity at most $k$ at $i$. If \eqref{eq:TCest-i} holds for every $i\in[r]$, that is if
\begin{equation}\label{eq:TCest}
\big| \Lambda_\Psi(f_1,\ldots,f_r)\big| \leq \min_{i\in [r]}\; \varepsilon(\|f_i\|_{U^{k+1}}),
\end{equation}
then we say that $\Psi$ has \emph{true complexity at most} $k$. We write $s(\Psi)$ for the least integer $k$ such that $\Psi$ has true complexity at most $k$. (Thus $s_{(i)}(\Psi)$ and $s(\Psi)$ may depend on $p$.)
\end{defn}
\noindent  We shall refer to a bound of the form \eqref{eq:TCest-i} for $k=s_{(i)}(\Psi)$ as a \emph{true-complexity bound at} $i$ for $\Psi$, and to a bound of the form \eqref{eq:TCest} for $k=s(\Psi)$ as a \emph{true-complexity bound} for $\Psi$. 

By Theorem \ref{thm:basicGT}, we always have $s(\Psi)\leq s_{\textup{CS}}(\Psi)$. However, as mentioned above, there are systems $\Psi$ for which this inequality is strict. Initial examples of this phenomenon were given in \cite{G&W1}. In subsequent work on this topic, notably by Manners in \cite[\S 2]{Man}, it was shown that even in quite simple families of systems this phenomenon is actually generic. 

In \cite{G&W1}, Gowers and Wolf formulated a conjecture giving an algebraic description of true complexity, purely in terms of the system $\Psi$ (without reference to the uniformity norms). In the integer setting, this conjecture was first settled by Green and Tao  for systems of linear forms satisfying a condition called the \emph{flag property} (see the update \cite{GTarith2} to the paper \cite{GTarith}), and the conjecture was then completely settled recently by Altman in \cite{Altman}. Over $\mb{F}_p$, the conjecture states that the true complexity of a system $\Psi\in \mb{F}_p^{r\times d}$ is equal to the smallest integer $k$ such that the tensor powers\footnote{For any $v\in \mb{F}_p^d$ and $m\in \mb{N}$, the tensor power $v^m$ is the vector $\big(v_{j_1}\cdots \,v_{j_m}:j_1,\ldots,j_m\in [d]\big)\in \mb{F}_p^{d^m}$.} $\psi_1^{k+1},...,\psi_r^{k+1}$ are linearly independent over $\mb{F}_p$. Gowers and Wolf made progress toward the conjecture in several papers \cite{G&W1,G&W2,G&W3}, eventually proving it for sufficiently large values of $p$ in \cite[Theorem 6.1]{G&W3}. This direction was later pursued further by other authors, leading to the result of Hatami, Hatami and Lovett \cite[Theorem 3.17]{HHL}, proving a general form of the  conjecture which settles essentially all its cases of interest in the finite field setting.

An important aspect of the above-mentioned results settling the Gowers--Wolf conjecture is that they were proved using results from higher-order Fourier analysis involving the inverse theorem for the uniformity norms. Because of the bounds in this theorem, its use in these results produces true-complexity bounds where the functions $\varepsilon(\delta)$ in \eqref{eq:TCest} have a poor (i.e.\ slow) decay rate, in particular slower than polynomial in $\delta$. This led to the following question, closely related to \cite[Problem 7.8]{G&W3}, and posed (in essentially equivalent form) in \cite[Question 1.4]{Man}.
\begin{question}[\cite{Man}]\label{Q:TC}
For systems of finite true complexity, can a true-complexity bound always be proved using only finitely many applications of the Cauchy-Schwarz inequality and changes of variables?
\end{question}
\noindent This question is important for quantitative applications, because true-complexity bounds proved in the ``elementary" way described in the question are of much better quality (polynomial in $\delta$) than those proved using the inverse theorem for the $U^k$ norms. In \cite{Man}, Manners initiated the investigation of this  question focusing on systems of six linear forms in three variables, this being the simplest setting where systems can have true complexity strictly smaller than their Cauchy-Schwarz complexity (see \cite[\S 2]{Man}). In particular, Manners answered the  question positively for any such system having true complexity 1, thus obtaining polynomial true-complexity bounds for such systems (see \cite[Theorem 1.5]{Man}). 

\begin{remark}
Soon after the completion of this work and its dissemination in \cite{CGSS-seq-CS-Eurocomb}, independent breakthrough work of Manners appeared in \cite{MannersTC} giving a full affirmative answer to Question \ref{Q:TC}. Our main results, discussed below, are geared towards applications in \cite{CGS} concerning the inverse theorem for Gowers norms over finite fields, and are thus much more specific (but also much less technical) than those in \cite{MannersTC}. See also Remark \ref{rem:quant} regarding quantitative aspects of these results in relation to \cite{MannersTC}. 
\end{remark}

\subsection{Statements of main results}\hfill\\
In this paper we introduce the following definition of complexity, which refines Cauchy-Schwarz complexity and yields a positive answer to Question \ref{Q:TC} for some classes of systems for which such an answer is not accessible using only Cauchy-Schwarz complexity.

Given a set $S\subset \mb{F}_p^d$, we write $\langle S\rangle$ for the vector subspace of $\mb{F}_p^d$ generated by $S$, and we write $S^c$ for the complement $\mb{F}_p^d\setminus S$.

\begin{defn}\label{def:seqCS}
Let $\Psi=\{\psi_1,\ldots,\psi_r\}$ be a system of linear forms over $\mb{F}_p$. For $k,\ell\in \mb{N}$, we say that $\Psi$ has \emph{sequential CS-complexity at most $(k,\ell)$ at $i\in [r]$} if there is a sequence $(\psi^{(j)}:=\psi_{i_j})_{j\in [\ell]}$ in $\Psi$, with $\psi^{(\ell)}=\psi_i$, such that for every $j\in [\ell]$ there exist sets $C_{j,1},\ldots,C_{j,k+1}\subset \Psi$ such that $\Psi\setminus \{\psi^{(1)},\ldots,\psi^{(j)}\}\subset \bigcup_{t=1}^{k+1} C_{j,t}$ and $\{\psi^{(1)},\ldots,\psi^{(j)}\}\subset \langle C_{j,t}\rangle^c$ for every $t\in [k+1]$. We say that $\Psi$ has \emph{sequential CS-complexity at most $(k,\ell)$} if it has sequential CS-complexity at most $(k,\ell)$ at every $i\in [r]$.
\end{defn}
\noindent 
Note that the case $\ell=1$ of sequential CS-complexity is the CS-complexity notion of Green and Tao. The main result of this paper is the following.

\begin{theorem}\label{thm:main}
Let $\Psi$ be a system of $r$ linear forms over $\mb{F}_p$ and suppose that $\Psi$ has sequential CS-complexity at most $(k,\ell)$ at $i\in [r]$. Then, for every collection of 1-bounded functions $f_1,\ldots,f_r:\mb{F}_p^n\to \mb{C}$, we have
\begin{equation}\label{eq:main}
\big| \Lambda_\Psi (f_1,\ldots,f_r) \big|\le \|f_{i}\|_{U^{k+1}}^{ 2^{1-\ell}}.
\end{equation}
\end{theorem}
\noindent Theorem \ref{thm:basicGT} is the case $\ell=1$ of Theorem \ref{thm:main}. We prove Theorem \ref{thm:main} in Section \ref{sec:mainrespf}.

Our main applications concern translation invariant systems.
\begin{defn}\label{def:transinv}
A system $\Psi\in \mb{F}_p^{r\times d}$ is \emph{translation invariant} if the image $\tIm(\Psi)=\Psi(\mb{F}_p^d)$ is a translation-invariant set, meaning that for every $y=(y_1,\ldots,y_r)\in \tIm(\Psi)$ and $\lambda\in\mb{F}_p$ we have $(y_1+\lambda,\ldots,y_r+\lambda)\in \tIm(\Psi)$.
\end{defn}
\noindent For $k,M\in \mb{N}$, we define the set $S_{k,M} :=\{z\in [0,p-1]^M: z_1+\cdots+z_M < k\}$, where addition is performed in $\mb{Z}$. The set $S_{k,M}$ can be viewed as a subset of $\mb{F}_p^M$ via the usual identification of $\mb{F}_p^M$ with $[0,p-1]^M$. We can then define the following system of linear forms in $M+1$ variables:
\begin{equation}\label{eq:PhikMdef}
\Phi_{k,M} := \big\{\phi_z(x,t_1,\ldots,t_M):=x+z_1t_1+\cdots +z_M t_M \;|\; z=(z_1,\ldots,z_M)\in S_{k,M}\big\}.
\end{equation}
For $k>M(p-1)+1$ we have $S_{k,M}=S_{M(p-1)+1,M}=[0,p-1]^M$, so $\Phi_{k,M}=\Phi_{M(p-1)+1,M}$. Therefore we can assume without loss of generality that $k\leq M(p-1)+1$. 

These systems $\Phi_{k,M}$ are translation invariant. The corresponding configurations are natural multivariable generalizations of arithmetic progressions; indeed $\Phi_{k,1}$ corresponds to arithmetic progressions of length $k$. As is well-known in this area, for $k\le p$ the system $\Phi_{k,1}$ has true complexity $k-2$, equal to its Cauchy-Schwarz complexity, with an optimal true-complexity bound given by Gowers's estimate \eqref{eq:GVNorig}. The question of how these facts might extend to $\Phi_{k,M}$ for $M>1$ turns out to be an interesting one. 

In Section \ref{sec:transinvapp} we first show that for every $M,p$ and $k\leq M(p-1)+1$, the system $\Phi_{k,M}$ has true complexity at least $k-2$; see Proposition \ref{prop:TrueCompSk}. It turns out that this true complexity is also \emph{at most} $k-2$ and that this can be proved with polynomial true-complexity bounds, but for $M>1$ this requires going beyond Cauchy-Schwarz complexity, unlike in the case $M=1$. The reason for this is that, for $M>1$, as soon as we move into the case $k> p$, it is no longer true in general that the Cauchy-Schwarz complexity of $\Phi_{k,M}$ is $k-2$. We illustrate this using results on hyperplane coverings over $\mb{F}_p$; see Remark \ref{rem:highchar}, Proposition \ref{prop:BallSerraApp} and Corollary  \ref{cor:CS>TC}.
However, using \emph{sequential} Cauchy-Schwarz complexity instead, and applying Theorem \ref{thm:main}, we obtain the following result, which extends Gowers's estimate \eqref{eq:GVNorig} to all systems $\Phi_{k,M}$.
\begin{theorem}\label{thm:GGVN}
Let $p$ be a prime, let $M\in \mb{N}$, and let $k\in [M(p-1)+1]$. Then the system $\Phi_{k,M}$ has true complexity $k-2$. Moreover, there is a constant $c=c_{k,M,p}\in (0,1]$ such that for every collection of $1$-bounded functions $(f_z:\mb{F}_p^n\to\mb{C})_{z\in S_{k,M}}$ we have
\begin{equation}\label{eq:GGVN}
\big| \Lambda_{\Phi_{k,M}}\big( (f_z)_{z\in S_{k,M}} \big)\big| \leq  \min_{z\in S_{k,M}} \|f_z\|_{U^{k-1}}^c.
\end{equation}
We can take $c=1$ for $k\leq p$ and $c=2^{1-|S_{k,M}|}$ for $k>p$.
\end{theorem}
\begin{remark}\label{rem:quant}
An important aspect here is that the constant $c$ is independent of the dimension $n$. It is likely that this constant can be improved; see Remark \ref{rem:constant}, which describes improvements in certain cases. Note however that, as observed in \cite{Man}, true-complexity bounds modulo $p$ can be unavoidably dependent on $p$ (i.e.\ there are systems for which the function $\varepsilon$ in \eqref{eq:TCest} is an unavoidably increasing function of $p$). Applying the recent general result of Manners \cite[Theorem 1.1.5]{MannersTC} here yields $c=2^{-M}$ where $M\ll |S_{k,M}|^3(\log (|S_{k,M}|)+\log\log(10 p))$.
\end{remark}
\noindent The above examples of systems with sequential CS-complexity at most $(k,\ell)$ and CS-complexity greater than $k$ occur for $k>p$, but note that this phenomenon can also occur for systems in the case $k\leq p$; we give an example in Remark \ref{rem:HCeg}. Such examples show that the refined control on $\Psi$-averages offered by Theorem \ref{thm:main} can be useful not only in the low characteristic case of the finite field setting, but also in the high characteristic case, and thus in the integer setting. On the other hand, we also have examples of systems showing that Theorem \ref{thm:main} does not yield directly polynomial true-complexity bounds for all finite-complexity systems; see Remark \ref{rem:scs-notsuff}.

Finally, let us mention another application of the above results: in the separate paper \cite{CGS} we use Theorem \ref{thm:GGVN} as a central ingredient for an algebraic description of compact nilspaces which admit  strongly equidistributed nilspace morphisms from (the additive groups of) vector spaces $\mb{F}_p^n$. This description in turn is key to a new proof of the inverse theorem for Gowers norms on $\mb{F}_p^n$, and  applications in ergodic theory, given in \cite{CGS}.

\section{Proof of the main result}\label{sec:mainrespf}
\noindent In this section we prove Theorem \ref{thm:main}. The proof uses the following linear-algebraic fact.

\begin{proposition}\label{prop:basic-lin-algebra}
Let $U_1,U_2$ be subspaces of $\mb{F}_p^M$ and for $i\in\{1,2\}$ let $D_i$ be a subspace of $U_i$ such that $D_1\cap U_2=D_2\cap U_1=\{0\}$. Then $(D_1+D_2)\cap U_i=D_i$ for $i=1,2$.
\end{proposition}

\begin{proof}
We prove that $(D_1+D_2)\cap U_i=D_i$ with $i=1$ (the case $i=2$ is proved similarly). To see the inclusion $(D_1+D_2)\cap U_1\subset D_1$, let $t=d_1+d_2=u_1$ with $d_j\in D_j$ for $j=1,2$ and $u_1\in U_1$. Then $d_2=u_1-d_1\in D_2\cap U_1=\{0\}$, so $t=d_1\in D_1$. The opposite inclusion is clear.
\end{proof}
\noindent Let us call a sequence $(\psi^{(j)})_{j\in[\ell]}$ as in Definition \ref{def:seqCS} a \emph{witness} sequence for the sequential CS-complexity of $\Psi$ being at most $(k,\ell)$  at $i\in [r]$; that is, we have $\psi^{(\ell)}=\psi_i$ and for every $j\in [\ell]$ the set $\Psi\setminus\{\psi^{(1)},\ldots,\psi^{(j)}\}$ can be covered by $k+1$ subsets such that $\{\psi^{(1)},\ldots,\psi^{(j)}\}$ is included in the complement of the linear span of each of these subsets.

\begin{proof}[Proof of Theorem \ref{thm:main}]
We fix any value of $k$ and argue by induction on $\ell$ for this value of $k$. The base case $\ell=1$ is given by Theorem \ref{thm:basicGT}, since then $s_{\textup{CS}(i)}(\Psi)\leq k$.

By relabeling the forms in $\Psi$ if necessary, we may assume that $\psi_\ell=\psi_i$ and that the forms $\psi_1,\psi_2,\ldots,\psi_\ell$ form a witness sequence for the sequential Cauchy-Schwarz complexity of $\Psi$ being at most $(k,\ell)$ at $\ell$. Having relabeled the forms this way, our aim is to prove that
\begin{equation}\label{eq:proof-cs-seq-1}
\big| \mb{E}_{x_1,\ldots,x_d\in \mb{F}_p^n} f_1\big(\psi_1(x_1,\ldots,x_d)\big)\cdots f_r\big(\psi_r(x_1,\ldots,x_d)\big) \big| \leq \|f_\ell\|_{U^{k+1}}^{1/2^{\ell-1}},
\end{equation}
where $d$ is the number of variables of $\Psi$. Let us view the system $\Psi$ as a subset\footnote{A linear form $\psi$ is a linear map $\mb{F}_p^d\to \mb{F}_p$, $(z_1,\ldots,z_d)\mapsto a_1z_1+\cdots+a_dz_d$, so we may identify $\psi$ with the element $(a_1,\ldots,a_d)\in \mb{F}_p^d$. Thus a collection of linear forms can be identified with a subset of $\mb{F}_p^d$.} of $\mb{F}_p^d$.

Let $T\in \mb{F}_p^{d\times d}$ be an invertible matrix such that $\psi_1 T=(1,0^{d-1})$ (viewing $\psi_1$ as a horizontal vector), where   in general $0^m$ denotes the 0 vector in $\mb{F}_p^m$. The homomorphism $(\mb{F}_p^n)^d\to (\mb{F}_p^n)^d$, $(x_1,\ldots,x_d)\mapsto y=T(x_1,\ldots,x_d)$, where $y_i=T_{i1}x_1+\cdots +T_{id}x_d$ for $i\in [d]$, is invertible (with inverse $(y_1,\ldots,y_d)\mapsto T^{-1}(y_1,\ldots,y_d)$ for the inverse matrix $T^{-1}\in \mb{F}_p^{d\times d}$). Hence the average in \eqref{eq:proof-cs-seq-1} equals
\begin{equation}\label{eq:proof-cs-seq-2}
\mb{E}_{x_1,\ldots,x_d\in \mb{F}_p^n} f_1\big((\psi_1T)(x_1,\ldots,x_d)\big)\cdots f_r\big((\psi_rT)(x_1,\ldots,x_d)\big).
\end{equation}
Let us relabel the elements $\psi_j T$ as $\psi_j$ for all $j\in [r]$, and note that $\psi_1,\ldots,\psi_\ell$ is then still a witness sequence for the sequential CS-complexity being at most $(k,\ell)$ at $\ell$.

Since $\psi_1$ is now $(1,0^{d-1})$, the average in \eqref{eq:proof-cs-seq-2} can now be written
\begin{equation}\label{eq:proof-cs-seq-3}
\mb{E}_{x_1\in \mb{F}_p^n}\,f_1(x_1)\; \mb{E}_{x_2,\ldots,x_d\in \mb{F}_p^n}\,f_2(\psi_2(x_1,\ldots,x_d))\cdots f_r(\psi_r(x_1,\ldots,x_d)).
\end{equation}
Applying the Cauchy-Schwarz inequality to the average over $x_1$, we obtain the following upper bound for the squared modulus of \eqref{eq:proof-cs-seq-3}:
\begin{equation}\label{eq:proof-cs-seq-4}
    \mb{E}_{x_1,x_2,\ldots,x_d,x'_2,\ldots,x'_d\in \mb{F}_p^n}\prod_{j\in [2,r]} f_j(\psi_j(x_1,x_2\ldots,x_d)) \prod_{j\in [2,r]} \overline{f_j(\psi_j(x_1,x'_2\ldots,x'_d))}.
\end{equation}
We shall now define for every $j\in [2r-2]$ an element $\psi_j'\in \mb{F}_p^{2d-1}$ and a function $g_j:\mb{F}_p^n\to \mb{C}$ so as to rewrite \eqref{eq:proof-cs-seq-4} in the following form:
\begin{equation}\label{eq:proof-cs-seq-5}
    \mb{E}_{x_1,\ldots,x_{2d-1}\in \mb{F}_p^n}\prod_{j\in [2r-2]} g_j\big(\psi_j'(x_1,\ldots,x_{2d-1})\big).
\end{equation}
Letting $\psi_{j,t}$ denote the $t$-th entry of $\psi_j\in \mb{F}_p^d$, the required definitions are the following:
\[
\psi_j':=\begin{cases}
(\psi_{j+1,1},\ldots, \psi_{j+1,d},0^{d-1})\in \mb{F}_p^{2d-1}, &j\in [r-1],\\
(\psi_{j-r+2,1},0^{d-1},\psi_{j-r+2,2},\ldots, \psi_{j-r+2,d})\in \mb{F}_p^{2d-1}, &j\in [r,2r-2],
\end{cases}
\]
and
\[
g_j:=\begin{cases}
f_{j+1}, &j\in [r-1],\\
\overline{f_{j-r+2}}, & j\in [r,2r-2]
\end{cases}.
\]
Letting $\Psi'$ denote the system $\Psi'=\{\psi'_j:j\in [2r-2]\}$, we shall now prove that the sequence $\psi_1',\ldots,\psi_{\ell-1}'$ is a witness for the sequential CS-complexity of $\Psi'$ being at most $(k,\ell-1)$ at $\ell-1$. By induction on $\ell$, this will imply that \eqref{eq:proof-cs-seq-5} has modulus at most $\|g_{\ell-1}\|_{U^{k+1}}^{1/2^{\ell-2}}= \|f_\ell\|_{U^{k+1}}^{1/2^{\ell-2}}$, so \eqref{eq:proof-cs-seq-3} has modulus at most $\|f_\ell\|_{U^{k+1}}^{1/2^{\ell-1}}$, which will complete the proof. 

Fix any $j^*\in[\ell-1]$. Consider the following subspaces of $\mb{F}_p^{2d-1}$:
\[
U_1:=\{y\in \mb{F}_p^{2d-1}: y_{d+1}=\cdots=y_{2d-1}=0\},\quad  U_2:=\{y\in \mb{F}_p^{2d-1}: y_2 = \cdots =y_d=0\}.
\]
Note that $\Psi'\subset U_1\cup U_2$ and that $U_1\cap U_2 =\langle (1,0^{2d-2})\rangle$. Note also that no element $\psi_j'$ is in $U_1\cap U_2$, for otherwise the corresponding form $\psi_{s}\subset \mb{F}_p^d$ (identifying $\mb{F}_p^d$ with either $U_1$ or $U_2$ depending on whether $j$ is at most $r-1$ or greater) would satisfy $\langle \psi_s\rangle \ni\psi_1=(1,0^{d-1})$, contradicting the assumption that $\Psi\setminus \{\psi_1\}$ is covered by subsets that do not contain $\psi_1$ in their linear span. Letting $\vartheta$ denote the isomorphism $U_1\to \mb{F}_p^d$, $y\mapsto (y_1,\ldots,y_d)$, we have $\vartheta(U_1\cap \Psi')=\Psi\setminus \{\psi_1\}$. By our assumptions, the set $\Psi\setminus \{\psi_1,\ldots,\psi_{j^*+1}\}$ can be covered by subsets  $C_{1},\ldots,C_{k+1}$ such that $\{\psi_1,\ldots,\psi_{j^*+1}\}\subset \mb{F}_p^d\setminus \langle C_{t}\rangle$ for every $t\in [k+1]$. Furthermore, letting $C_{t}'=\vartheta^{-1}(C_{t})$, we have $(\Psi'\setminus \{ \psi_1',\ldots,\psi_{j^*}' \})\cap U_1\subset \bigcup_{t\in [k+1]} C'_{t}$ and $\langle C'_{t}\rangle \cap \langle (1,0^{2d-2})\rangle =\{0\}$ for every $t\in [k+1]$. This last equality follows from the fact that if $\langle C'_{t}\rangle$ contained $(1,0^{2d-2})$ for some $t\in [k+1]$, then $\langle \vartheta(C'_{t})\rangle=\langle C_{t}\rangle$ would contain $\vartheta((1,0^{2d-2}))=(1,0^{d-1})=\psi_1$, a contradiction. Similarly, using an isomorphism $U_2\to \mb{F}_p^d$ mapping $U_2\cap \Psi'$ onto $\Psi\setminus \{\psi_1\}$, we obtain a covering of $U_2\cap \Psi'$ by subsets $E'_{1},\ldots,E'_{k+1}$ such that $\langle E'_{t}\rangle \cap \langle (1,0^{2d-2})\rangle =\{0\}$ for all $t\in [k+1]$. Now for each $t\in [k+1]$ we let $D_{1,t}=\langle C'_{t}\rangle\subset U_1$ and $D_{2,t}=\langle E'_{t}\rangle\subset U_2$, and define the sets $S_t=C'_{t} \cup E'_{t}$. By construction, these sets $S_t$ cover $\Psi'\setminus\{\psi'_1,\ldots,\psi'_{j^*}\}$, and by Proposition \ref{prop:basic-lin-algebra}, each subspace $\langle S_t\rangle = D_{1,t} + D_{2,t}$ satisfies $\langle S_t\rangle \cap U_i = D_{i,t}$ for $i=1,2$, so it avoids the set $\{\psi'_1,\ldots,\psi'_{j^*}\}$.
\end{proof}

\section{Applications to translation-invariant systems}\label{sec:transinvapp}

\noindent It follows clearly from Definition \ref{def:transinv} that a system $\Psi\in \mb{F}_p^{r\times d}$ is translation invariant if and only if there is an invertible matrix $R\in \mb{F}_p^{d\times d}$ such that the matrix $\Psi':=\Psi R$ has first column equal to the vector $1^r:=(1,\ldots,1)\in \mb{F}_p^r$. Note also that changing $\Psi$ to $\Psi'$ does not affect the image (i.e.\ $\tIm(\Psi')=\tIm(\Psi)$), whence $\Lambda_\Psi(f_1,\ldots,f_r)=\Lambda_{\Psi'}(f_1,\ldots,f_r)$. We shall therefore assume from now on that for every translation invariant system the corresponding matrix $\Psi$ is given with first column equal to $1^r$.
 \begin{defn}\label{def:point-set}
Given a translation invariant system $\Psi\in\mb{F}_p^{r\times (M+1)}$ (with first column $1^r$), we define the \emph{associated set} of $\Psi$ to be 
\[
Z=\big\{a_i:= \psi_i|_{[2,M+1]}\in \mb{F}_p^M\;|\; i\in [r]\big\}.
\]
\end{defn}

Recall that an \emph{affine combination} of a set $Z\subset \mb{F}_p^M$ is a linear combination $\sum_{a\in Z}\lambda_a a$ such that $\sum_{a\in Z}\lambda_a =1$. Given a subset $X$ of a vector space, we denote by $\aff(X)$ the affine span of $X$, i.e.\ the minimal affine subspace that includes $X$.

For translation-invariant systems, the definition of Cauchy-Schwarz complexity can be rephrased as a geometric property of the associated sets. 
\begin{lemma}\label{lem:CSinvgeom}
Let $\Psi\in \mb{F}_p^{r\times (M+1)}$ be translation invariant and let $Z=\{a_1,\ldots,a_r\}\subset\mb{F}_p^M$ be the associated set. Then $s_{\textup{CS}(i)}(\Psi)\leq k$ if and only if $Z\setminus\{a_i\}$ can be covered by $k+1$ affine subspaces, each of which excludes $a_i$.
\end{lemma}
\begin{proof}
Note that for any set $C\subset [r]$, a form $\psi_i$ in $\Psi$ is in the linear span of $\{\psi_j:j\in C\}$ if and only if $a_i$ can be written as a linear combination $\sum_{j\in C}\lambda_j a_j$ such that $\sum_{j\in C}\lambda_j=1$ (the latter condition following from the first column of $\Psi$ being $1^r$ and $\sum_{j\in C}\lambda_j \psi_j=\psi_i$). Thus
\begin{equation}\label{eq:spanaff}
\psi_i\in \langle \{\psi_j:j\in C\}\rangle \;\Leftrightarrow\;a_i\in \aff(\{a_j: j\in C\}).
\end{equation}
If $s_{\textup{CS}(i)}(\Psi)\leq k$, then by definition we can cover $\Psi\setminus\{\psi_i\}$ by subsets $C_1,\ldots,C_{k+1}$ such that for every $t\in [k+1]$ we have $\psi_i\not\in \langle \{\psi_j:j\in C_t\}\rangle $. By \eqref{eq:spanaff} this implies that the affine subspace $H_t:=\aff (\{a_j: j\in C_t\})$ excludes $a_i$ for each $t\in [k+1]$, and we have $Z \setminus\{a_i\}\subset H_1\cup \cdots\cup H_{k+1}$.

Conversely, if $Z\setminus\{a_i\}$ is covered by affine subspaces $H_1,\ldots,H_{k+1}$ each excluding $a_i$, then letting $C_t:=\{j\in [r]: a_j\in H_t\cap (Z\setminus\{a_i\})\}$ for each $t\in [k+1]$, we have that $Z\setminus\{a_i\}$ is still covered by the affine spaces $\aff(\{a_j:j\in C_t\})$, $t\in [k+1]$, and then by \eqref{eq:spanaff} the sets $\{\psi_j:j\in C_t\}$ show that $s_{\textup{CS}(i)}(\Psi)\leq k$.
\end{proof}
\noindent Similarly, Lemma \ref{lem:seqCSinvgeom} below describes sequential CS-complexity for translation invariant systems as a geometric property of the associated sets, using the following terminology.

\begin{defn}[$k$-coverings excluding points]
Let $T\subset \mb{F}_p^n$ and let $a^{(1)},\ldots,a^{(\ell)}\in \mb{F}_p^n\setminus T$. We say that $T$ is \emph{$k$-coverable excluding} $a^{(1)},\ldots,a^{(\ell)}$ if there exist $k$ affine subspaces (not necessarily different) $V_1,\ldots,V_k\subset \mb{F}_p^n$ such that $T\subset \bigcup_{i=1}^k V_i$ and $\{a^{(1)},\ldots,a^{(\ell)}\}\subset \mb{F}_p^n\setminus (\bigcup_{i=1}^k V_i)$.
\end{defn}

\begin{lemma}\label{lem:seqCSinvgeom}
Let $\Psi\in \mb{F}_p^{r\times (M+1)}$ be a translation invariant system, with associated set $Z=\{a_1,\ldots,a_r\}\subset \mb{F}_p^M$. Then $\Psi$ has sequential CS-complexity at most $(k,\ell)$ at $i\in [r]$ if and only if there is a sequence $a^{(1)},\ldots,a^{(\ell)}\in Z$, with $a^{(\ell)}=a_i$, such that for every $j\in [\ell]$ the set $Z\setminus\{a^{(1)},\ldots,a^{(j)}\}$ is $(k+1)$-coverable excluding $a^{(1)},\ldots,a^{(j)}$.
\end{lemma}
\noindent We shall call a sequence $(a^{(j)})_{j\in [\ell]}$ in $Z$ with the property in this lemma a \emph{witness} sequence (for the sequential CS-complexity of $\Psi$ being at most $(k,\ell)$ at $i\in [r]$).

\begin{proof}
The proof is similar to that of Lemma \ref{lem:CSinvgeom}, using \eqref{eq:spanaff} (and the assumption that the first column of $\Psi$ is the vector $1^r$).
\end{proof}
\noindent Let us now focus on the systems $\Phi_{k,M}$ defined in  \eqref{eq:PhikMdef}.

First we prove the following lower bound on the true complexity of $\Phi_{k,M}$.
\begin{proposition}\label{prop:TrueCompSk}
Let $p$ be a prime, let $M\in \mb{N}$, and let $k\in [M(p-1)+1]$. There is a strictly increasing sequence of integers $(n_\ell)_{\ell\in \mb{N}}$ with the following property: for every $\ell$ there exist 1-bounded functions $f_z:\mb{F}_p^{n_\ell}\to \mb{C}$ for $z\in S_{k,M}$, such that
\begin{equation}\label{eq:av}
\Lambda_{\Phi_{k,M}}\big( (f_z)_{z\in S_{k,M}}\big) = 1
\end{equation}
and $\|f_{0^M}\|_{U^{k-2}} \to 0$ as $\ell\to \infty$. Hence the system $\Phi_{k,M}$ has true complexity at least $k-2$.
\end{proposition}

\begin{proof}
We take $n_\ell=\ell M$. First let us suppose that $\ell=1$. Fix any $w\in [0,p-1]^M$ such that $w_1+\cdots +w_M = k-1$ (thus $w\in S_{k,M}$), and suppose without loss of generality that $w_1>0$. Let $P$ denote the  map $\mb{F}_p^M\to\mb{F}_p$ defined by $P(x)=\binom{|x_1|_p}{w_1-1}\binom{|x_2|_p}{w_2}\cdots \binom{|x_M|_p}{w_M}$, where $|x_i|_p\in[0,p-1]$ is the representative in $[0,p-1]$ of $x_i\in \mb{F}_p$. We have that $P$ is a polynomial map of degree $k-2$ (see e.g.\ \cite[Lemma 1.6 $(iii)$]{T&Z-Low}).

For any $x,t_1,\ldots,t_M\in \mb{F}_p^M$ let $\q=\q_{x,t_1,\ldots,t_M}$ denote the $\mb{F}_p^M$-valued cube of dimension $k-1$, defined for $v=(v_1,\ldots,v_{k-1})\in \{0,1\}^{k-1}$ by
\[
\q(v_1,\ldots,v_{k-1}):= x+(v_1+\cdots +v_{w_1})t_1+(v_{w_1+1}+\cdots +v_{w_1+w_2})t_2+\cdots+(v_{k-w_M}+\cdots +v_{k-1})t_M.
\]
For any such cube $\q$, letting $|z|:=z_1+\cdots+z_M$ for $z\in [0,p-1]^M$,  the Gray-code alternating sum $\sigma_{k-1}(P\co \q):=\sum_{v\in \{0,1\}^{k-1}} (-1)^{v_1+\cdots+v_{k-1}}P\co\q(v)$ satisfies
\[
\sigma_{k-1}(P\co \q) = \sum_{z\in \prod_{i=1}^M[0,w_i]} (-1)^{|z|}\binom{w_1}{z_1}\cdots \binom{w_M}{z_M}P(x+z_1 t_1+\cdots +z_M t_M).
\]
By \cite[Theorem 2.2.14]{Cand:Notes1}, since $P$ is a polynomial of degree at most $k-2$, it is a morphism\footnote{The term ``morphism" here is used in the category of nilspaces, see \cite[Definition 2.2.11]{Cand:Notes1}.} from $\mc{D}_1(\mb{Z}_p^M)$ to $\mc{D}_{k-2}(\mb{Z}_p)$,\footnote{See \cite[Definition 2.2.30]{Cand:Notes1} for the definition of the nilspace $\mc{D}_i(Z)$ associated with an abelian group $Z$, and \cite[Theorem 2.2.14]{Cand:Notes1} for an equivalent representation of morphisms between such nilspaces.} and therefore $P\co \q$ is a $(k-1)$-cube on $\mb{Z}_p$, so $\sigma_{k-1}(P\co \q)=0$ by \cite[Proposition 2.2.28]{Cand:Notes1}.

Now, for each $z\in S_{k,M}$, we define $f_z: \mb{F}_p^{n_1}\to \mb{C}$, $y\mapsto e_p((-1)^{|z|}\binom{w_1}{z_1}\cdots \binom{w_M}{z_M}P(y))$ where $e_p(t)=e^{2 \pi i t/p}$ for $t\in \mb{F}_p$. Then
\begin{eqnarray*}
\Lambda_{\Phi_{k,M}}\big( (f_z)_{z\in S_{k,M}}\big) & = & \mb{E}_{x,t_1,\ldots,t_M\in \mb{F}_p^M}\prod_{z\in S_{k,M}} f_z(x+z_1t_1+\cdots+z_Mt_M) \\
& = & \mb{E}_{x,t_1,\ldots,t_M\in \mb{F}_p^M}\; e_p\big(\sigma_{k-1}(P\co \q_{x,t_1,\ldots,t_M})\big) = 1.
\end{eqnarray*}
We claim that $\|f_{0^M}\|_{U^{k-2}(\mb{F}_p^M)}< 1$. To see this recall that, letting $\nabla_h$ denote the difference operator defined by $\nabla_h P(x)=P(x+h)-P(x)$, we have
\[
\|f_{0^M}\|_{U^{k-2}}^{2^{k-2}}=\mb{E}_{x,h_1,\ldots,h_{k-2}\in \mb{F}_p^M} \; e_p(\nabla_{h_1}\cdots \nabla_{h_{k-2}}P(x)).
\]
This is an average of complex numbers of modulus 1 which include the number 1 (indeed any term in the average is 1 if it involves some $h_i=0$). Therefore, the claim will follow if we show that the average also includes numbers different from 1. For $h_1 = \cdots = h_{w_1-1} = (1,0^{M-1})$, $h_{w_1}=\cdots = h_{w_1+w_2-1} = (0,1,0^{M-2})$, $\ldots$, $h_{k-w_M-1}=\cdots=h_{k-2}= (0^{M-1},1)$, the corresponding term  $e_p\big(\nabla_{h_1}\cdots \nabla_{h_{k-1}}P(x)\big)$ equals\footnote{Note that in general $\nabla_1 \binom{x}{k}=\binom{x+1}{k}-\binom{x}{k}=\binom{x}{k-1}$ and so $\nabla_1^k \binom{x}{k}=1$.}
\[
e_p\Big(\big(\nabla_1^{w_1-1}\binom{|x_1|_p}{w_1-1}\big)\big(\nabla_1^{w_2}\binom{|x_2|_p}{w_2}\big)\cdots \big(\nabla_1^{w_M}\binom{|x_M|_p}{w_M}\big)\Big)= e_p(1)\neq 1,
\]
and our claim follows.

For general $\ell>1$, let us define a collection of functions $(f'_z:\mb{F}_p^{n_\ell}\to\mb{C})_{z\in S_{k,M}}$ by taking products of the above functions $f_z$. More precisely, for any element $y\in \mb{F}_p^{n_\ell}=\mb{F}_p^{\ell M}$,  denoting this by $(y^{(1)},\ldots,y^{(\ell)})$ where $y^{(j)}\in \mb{F}_p^M$,  we define the function $f'_z$ by $f'_z(y^{(1)},\ldots,y^{(\ell)}):= f_z(y^{(1)})\cdots f_z(y^{(\ell)})$. We then have $\mb{E}_{x,t_1,\ldots,t_{M}\in \mb{F}_p^{n_\ell}}\prod_{z\in S_{k,M}} f'_z(x+z\cdot t) = 1$ and $\|f'_z\|_{U^{k-2}(\mb{F}_p^{\ell M})} = \|f_z\|_{U^{k-2}(\mb{F}_p^M)}^\ell$, so $\|f'_{0^M}\|_{U^{k-2}(\mb{F}_p^{\ell M})}\to 0$ as $\ell\to \infty$.
\end{proof}

\begin{remark}\label{rem:highchar}
For $k\leq p$ the Cauchy-Schwarz complexity of $\Phi_{k,M}$ at $0^M$ is  $k-2$, because we can use the $k-1$ non-zero hyperplanes (that is,  hyperplanes not containing $0^M$) $H_j=\{z\in\mb{F}_p^M: z_1+\cdots+z_M=j\!\!\mod p\}$, $j\in [k-1]$ to cover $S_{k,M}\setminus\{0^M\}$. However, this simple argument breaks down for $k> p$, because the hyperplane $H_p$ contains $0^M$ so we cannot use it as part of a suitable cover of $S_{p+1,M}\setminus\{0^M\}$. One might believe at first that $S_{p+1,M}\setminus\{0^M\}$ could be covered some other way by $p$ non-zero hyperplanes. For instance, by inspection it can be seen that for $p=3$ and $M=2$ the set $S_{4,2}\setminus\{0^2\}\subset \mb{F}_3^2$ can be covered by $3$ non-zero hyperplanes (lines not containing $0^2$). However, as the next result shows, for $p\geq 5$ this is no longer possible. This will show that the true complexity of $\Phi_{k,M}$ can be strictly smaller than its Cauchy-Schwarz complexity when $k> p$, which motivates refining Cauchy-Schwarz complexity in order to prove Theorem \ref{thm:GGVN}.
\end{remark}

\begin{proposition}\label{prop:BallSerraApp}
For $p\geq 5$, the number of non-zero lines in $\mb{F}_p^2$ needed to cover the set $S_{p+1,2}\setminus\{0^2\}$ is at least $p+1$.
\end{proposition}
\begin{proof}
Suppose for a contradiction that $L_1,\ldots,L_p$ are non-zero lines covering $S_{p+1,2}\setminus\{0^2\}$. 

We claim that $\bigcup_{i\in [p]} L_i$ must also contain a point $z\in \mb{F}_p^2$ with $|z_1|_p+|z_2|_p=p+1$. If this holds then $\bigcup_{i\in [p]} L_i$ covers the set $B=\big([0,|z_1|_p]\times [0,|z_2|_p]\big)\setminus\{0^2\} \subset (S_{p+1,2}\cup\{z\})\setminus\{0^2\}$. This then yields a contradiction since, by known results on  covering by hyperplanes, we need at least $p+1$ non-zero lines to cover $B$ (see for instance \cite[Theorem 5.3]{B&S} applied with $t=1$ and $D_1=D_2=\{0\}$; see also \cite{A&F}). Hence it suffices to prove the claim.

To prove the claim, assume for a contradiction that for every $i\in [p]$ the line $L_i$ contains no point $z\in \mb{F}_p^2$ with $|z_1|_p+|z_2|_p=p+1$. Note that each of the points on the line segment $\{z\neq 0^2: |z_1|_p+|z_2|_p=p\}$ must be contained in some line $L_i$ (since this segment is a subset of $S_{p+1,2}\setminus\{0^2\}$), and no line $L_i$ can contain more than one point on this segment (otherwise this line $L_i$ would be the subspace $\{z: z_1+z_2=0\!\!\mod p\}$, contradicting that $L_i$ is a non-zero line). Therefore, there are at least $p-1$ lines among the lines $L_i$ needed to cover this segment. Since each such line $L_i$ is not parallel to the segment, this line $L_i$ must meet the line $\{z: z_1+z_2=1\mod p\}$ at some point. But this point cannot be on the part $\{z: |z_1|_p+|z_2|_p=p+1\}$ of this line, by our initial assumption. Hence each of these lines $L_i$ must contain either the point $e_1=(1,0)$ or the point $e_2=(0,1)$. It follows that without loss of generality there are at least $(p-1)/2$ lines $L_i$ containing $e_1$. But then there remain only $(p+1)/2$ lines among the $L_i$ to cover the remaining $p-2$ points on the $x$-axis (other than $e_1$ and the origin). This yields a contradiction if $(p+1)/2<p-2$, i.e. if $p>5$. If $p=5$ then the only possibility is if, of the five lines $L_i$, at most two go through $e_1$ and at most two go through $e_2$ (indeed, if three lines go through $e_1$ then the other 2 lines cannot cover the remaining 3 points on the $x$-axis). But then there is a line $L_i$ containing neither $e_1$ nor $e_2$, and so,  in order for this line not to contain any $z$ with $|z_1|_p+|z_2|_p=p+1$, it must be parallel to the line $\{z: z_1+z_2 = 1\mod p\}$. But since there can be at most one such line $L_i$ among the five (since the other four lines must cover the segment $\{z: |z_1|_p+|z_2|_p=5\}$), we are now in a quite restricted situation where there is precisely one such line, say $L_1$, parallel to $\{z: |z_1|_p+|z_2|_p=5\}$, and then 2 lines, say $L_2,L_3$, through $e_1$, and the other two lines $L_4,L_5$ through $e_2$. In this case it is seen by inspection that the set $S_{5,2}\setminus\{0^2\}$ is not covered by the five lines $L_i$.
\end{proof}
\begin{corollary}\label{cor:CS>TC}
For $p\geq 5$ and $M\geq 2$ we have $s_{\textup{CS}}(\Phi_{p+1,M})\geq p$.
\end{corollary}
\begin{proof}
We argue by induction on $M$. The case $M=2$ is given by Proposition \ref{prop:BallSerraApp}. For $M\geq 3$, if there is a covering of $S_{p+1,M}\setminus\{0^M\}$ by non-zero hyperplanes in $\mb{F}_p^M$, then none of the hyperplanes can be the subspace $\{z_M=0\}$. Thus, each of these hyperplanes has an intersection with this subspace which is either empty (which happens if the hyperplane is parallel to $\{z_M=0\}$) or is included in a non-zero hyperplane inside this subspace $\{z_M=0\}$. The union of these intersections must cover $(S_{p+1,M}\cap \{z_M=0\})\setminus \{0^M\}$, so by the case $M-1$ there must be at least $p+1$ such hyperplanes.
\end{proof}

\begin{remark}
The proof of Proposition \ref{prop:BallSerraApp} can be generalized without much difficulty to strengthen the result a bit further. For instance, it can be proved this way that for any prime $p>19$, at least $p+2$ non-zero lines are needed to cover the set $S_{p+2,2}\setminus\{0^2\}$ in $\mb{F}_p^2$, and so $s_{\textup{CS}}(\Phi_{p+2,M})\geq p+1$. We omit the details. It would be interesting to obtain lower bounds on $s_{\textup{CS}}(\Phi_{k,M})$ for large primes $p$ and higher values of $k>p$, as this problem is related to hyperplane-covering problems and may require interesting refinements of results such as those in \cite{A&F,B&S}.
\end{remark}

\noindent We now proceed to the main task of this section, namely proving Theorem \ref{thm:GGVN}, and especially deducing \eqref{eq:GGVN} from Theorem \ref{thm:main}. For this we shall use the following result, which gives us a sequence that goes through all of $S_{k,M}$ and from which we will be able to take initial segments to obtain adequate witness sequences for any $z\in S_{k,M}$.

\begin{proposition}\label{prop:covering-s-k}
Let $p$ be a prime and let $M$ and $k$ be positive integers. There exists a sequence $a^{(1)},\ldots,a^{(|S_{k,M}|)}$ such that $\{a^{(i)}\}_{i=1}^{|S_{k,M}|}=S_{k,M}$, with $a^{(|S_{k,M}|)}=0^M$, and such that for all $i\in [\,|S_{k,M}|\,]$ the set $S_{k,M}\setminus \{a^{(1)},\ldots,a^{(i)}\}$ is $(k-1)$-coverable excluding $\{a^{(1)},\ldots,a^{(i)}\}$.
\end{proposition}

\begin{proof}
We argue by induction on $M$.

For $M=1$, we have two cases, according to whether $k<p$ or $k\geq p$. If $k<p$, then $S_{k,1}=[0,k-1]$, and then, setting $a^{(i)}:=k-i$ for $i\in [k]$, it is readily seen that $S_{k,1}\setminus \{a^{(1)},\ldots,a^{(i)}\}$ is $(k-1)$-coverable excluding $\{a^{(1)},\ldots,a^{(i)}\}$ for each $i\in [k]$ (we take singletons as the affine subspaces; recall that repetition of the covering subspaces is allowed). If $k\geq p$, then $S_{k,1}=[0,p-1]$, but then we can set  $a^{(i)}:=p-i$ for $i\in [p]$, and we then see again that $S_{k,1}\setminus \{a^{(1)},\ldots,a^{(i)}\}$ is $(k-1)$-coverable excluding $\{a^{(1)},\ldots,a^{(i)}\}$ for each $i\in[p]$ (using singletons again).

For $M>1$, we can assume by induction that the case $M-1$ holds. For $j\in [0,p-1]$ let $P_j = \{(t_1,\ldots,t_M)\in [0,p-1]^M: t_M=j\}$, and let $r_j:=|S_{k,M}\cap P_j|$. We shall construct an appropriate sequence $(a^{(i)})_{i\in [\ell]}$ by first constructing its initial segment $a^{(1)},\ldots,a^{(r_{p-1})}$ covering $S_{k,M}\cap P_{p-1}$, then continuing with the next segment $a^{(r_{p-1}+1)},\ldots,a^{(r_{p-1}+r_{p-2})}$ covering $S_{k,M}\cap P_{p-2}$, and so on, until we end up with a sequence covering all of $S_{k,M}$, with final term $a^{(\ell)}=0^M$, with $\ell=r_{p-1}+r_{p-2}+\cdots+r_0 = |S_{k,M}|$. 

We begin with $P_{p-1}$, noting the following fact, which enables us to use the case $M-1$:
\[
S_{k,M}\cap P_{p-1}=\{ t\in [0,p-1]^M: (t_1,\ldots,t_{M-1})\in S_{k-p+1,M-1}, t_M= p-1\}.
\]
Let $b^{(1)},\ldots,b^{(r_{p-1})}$ be a sequence of distinct points in $S_{k-p+1,M-1}$ with $r_{p-1}=|S_{k-p+1,M-1}|$, given by our inductive assumption in the case $M-1$. Note that if $k< p$ then $S_{k,M}\cap P_{p-1}=\emptyset$ and the sequence $b^{(1)},\ldots,b^{(r_{p-1})}$ is empty. Let $a^{(i)}:=(b^{(i)},p-1)\in [0,p-1]^M$ for $i\in [r_{p-1}]$. We claim that $a^{(1)},\ldots,a^{(r_{p-1})}$ form an appropriate initial segment of our desired sequence. To prove this we need to show that $S_{k,M}\setminus\{a^{(1)},\ldots,a^{(i)}\}$ is $(k-1)$-coverable excluding $\{a^{(1)},\ldots,a^{(i)}\}$ for each $i\in [r_{p-1}]$.
We know that for every such $i$ there exist affine subspaces $H_1,\ldots,H_{k-p}$ in $[0,p-1]^{M-1}$ that cover $S_{k-p+1,M-1}\setminus\{b^{(1)},\ldots,b^{(i)}\}$ excluding $\{b^{(1)},\ldots,b^{(i)}\}$. Consider the affine subspaces $H_i'$ of $\mb{F}_p^M$ defined by $H'_i:=\{(t_1,\ldots,t_M)\in[0,p-1]^M:(t_1,\ldots,t_{M-1})\in H_i, t_M=p-1\}$, $i\in [k-p]$. Since
\[
(S_{k,M}\setminus\{a^{(1)},\ldots,a^{(i)}\})\cap P_{p-1} = \big(S_{k-p+1,M-1}\setminus\{b^{(1)},\ldots,b^{(i)}\}\big)\times \{p-1\},
\]
it is clear that the $H_i'$ cover $(S_{k,M}\setminus\{a^{(1)},\ldots,a^{(i)}\})\cap P_{p-1}$ excluding $\{a^{(1)},\ldots,a^{(i)}\}$. To complete our covering, we just take the hyperplanes $P_j$ for $j\in [0,p-2]$. In total we thus have $k-p$ affine subspaces $H_i'$ and $p-1$ affine hyperplanes $P_j$. Hence we have found a cover of $S_k^M\setminus\{a^{(1)},\ldots,a^{(i)}\}$ excluding $\{a^{(1)},\ldots,a^{(i)}\}$ by $k-1$ affine subspaces, as required.

We now construct the second segment of our sequence similarly, by  covering sequentially the points of $S_{k,M}\cap P_{p-2}$. In this case, by induction we will have $k-p+1$ affine subspaces coming from applying the case $M-1$ to $S_{k-p+2,M-1}$, and $p-2$ hyperplanes that will cover the points $t\in S_{k,M}$ such that $t_M<p-2$. Again, the total number of affine subspaces in our cover is $k-1$, as required. 

We can continue this process for $S_{k,M}\cap P_{p-3}$ and so on, down to $S_{k,M}\cap P_0$, where we will end up with the final point $a^{(\ell)}=0^M$. The result follows.
\end{proof}

\begin{proof}[Proof of Theorem \ref{thm:GGVN}]
For $k\le p$ the result follows from Theorem \ref{thm:basicGT}. For $k> p$,  Proposition \ref{prop:covering-s-k} tells us that for any element $z\in [0,p-1]^M$ we can find a witness sequence for the sequential CS-complexity being at most $(k-2,\ell_{z,k,M,p})$ at $z$, with the bound $\ell_{z,k,M,p}\le |S_{k,M}|$ for all $z\in S_{k,M}$ (equality in this bound is attained for $z=0^M$ in our construction). The result follows by Theorem \ref{thm:main}.
\end{proof}

\begin{remark}\label{rem:constant}
It would be interesting to improve the constant $c$ in Theorem \ref{thm:GGVN}. To do so using sequential CS-complexity could involve determining the smallest $\ell$ such that $\Phi_{k,M}$ has sequential CS-complexity at most $(k-2,\ell)$ at $0^M$. In certain special cases we can find shorter witness sequences than the one used in the proof of Proposition \ref{prop:covering-s-k}. For example, for $p=5$, $k=6$ and $M=2$, it is not hard to see that $\{1^2,0^2\}$ is a witness sequence for sequential-CS-complexity being at most $(4,2)$ at $0^2$. For $M=2$ and general $p < k$, it can be proved that the points on the line $t_1+t_2=0$ form a witness sequence at $0^2$, to obtain by Theorem \ref{thm:main} that $\big| \Lambda_{\Phi_{k,2}}\big((f_z)_{z\in S_{k,2}}\big)\big| \leq \|f_{0^2}\|_{U^{k-1}}^{2^{1-p}}$.
\end{remark}

\begin{remark}\label{rem:HCeg}
Consider the translation invariant system
\[
\Psi=\{x+y,\; x+z,\; x+2z, \; x+y+3z, \; x+2y+3z, \; x+3y+3z\}.
\]
For $p\geq 7$, this system can be seen to have CS complexity at least $2$. Indeed, the associated set is $Z=\{(1,0), (0,1), (0,2), (1,3), (2,3), (3,3)\}$, and it can be checked by inspection that $Z\setminus \{(3,3)\}$ cannot be covered by two affine lines. On the other hand, the sequential CS complexity of this system is at most $(1,2)$ at each point $P\in Z$. To see this we just need to exhibit a witness sequence of length at most 2  ending at $P$. For $P=(1,0)$ it is clear that $Z\setminus \{P\}$ can be covered by two lines, so in this case we even have a witness sequence of length 1. For $P\in Z\setminus\{(1,0)\}$, it is readily seen that $\{(1,0),P\}$ is a valid witness sequence, because we have already seen that  $Z\setminus\{(1,0)\}$ can be suitably covered by two lines, and then $Z\setminus\{(1,0),P\}$ consists of four points which can also always be covered by two lines avoiding $\{(1,0),P\}$.
\end{remark}

\begin{remark}\label{rem:scs-notsuff}
There are systems of finite true complexity for which sequential CS-complexity cannot yield true-complexity bounds via Theorem \ref{thm:main}. For example, it can be checked that, for $p\ge 23$, the translation invariant system
\[
\Psi=\{x,\; x+y,\; x+z,\; x+10y+z,\; x+y+2z,\; x+2y+2z\}
\]
has true complexity 1 but, for every $\ell$, no sequence of points in $\mb{F}_p^3$ can be a witness sequence of sequential CS-complexity of $\Psi$ being at most $(1,\ell)$. To prove this, consider the associated set for $\Psi$, namely $Z=\{(0,0),(1,0), (0,1),(10,1),(1,2),(2,2)\}$, and recall that the first step to find a witness sequence for $\Psi$ would be to find a point $P\in Z$ such that $Z\setminus \{P\}$ can be covered with two affine lines in $\mb{F}_p^2$. If this were possible, then there would be three points in $Z$ coverable by a single line. But a simple computation shows that for $p\ge 23$ no three points of $Z$ are in the same affine line. Note also that this particular example is covered by \cite[Theorem 1.5]{Man}, so there is a polynomial true-complexity bound for this system $\Psi$. It would be interesting to know if there are more refined uses of sequential CS-complexity, possibly combining it with the methods from \cite{Man,MannersTC} (or further refinements of the notion of sequential CS-complexity itself), yielding better true-complexity bounds than those obtained in this paper or in \cite{MannersTC}.
\end{remark}

\section*{Acknowledgements}
\noindent We thank the anonymous referees for useful comments that helped to improve this paper.

\section*{Funding}
\noindent All authors received funding from Spain's MICINN project PID2020-113350GB-I00. The second-named author received funding from projects KPP 133921 and
Momentum (Lend\"u- let) 30003 of the Hungarian Government. The research was also supported partially by the NKFIH ``\'Elvonal'' KKP 133921 grant and partially by the Hungarian Ministry of Innovation and Technology NRDI Office within the framework
of the Artificial Intelligence National Laboratory Program.

\end{document}